\documentclass[10pt]{amsart}
\title [ Optimal Convergence Implies Numerical Smoothness]
{Optimal Order Convergence Implies Numerical Smoothness }
\author{So--Hsiang Chou}
\thanks{Department of Mathematics and Statistics,
Bowling Green State University, Bowling Green, OH, 43403-0221,
USA}
\date{\today}

\usepackage{amsmath,amssymb,amsfonts,amsxtra,amsthm}
\usepackage{graphicx,subfigure,epsfig}
\usepackage{stmaryrd}

\newtheorem{theorem}{Theorem}[section]
\newtheorem{definition}[theorem]{Definition}
\newtheorem{lemma}[theorem]{Lemma}
\newtheorem{remark}[theorem]{Remark}
\newcommand{\scp}{\mathcal P}
\newcommand{\ps}{\mathbb P}

\newcommand{\norm}{\|}
\newcommand{\bx}{\mathbf x}

\newcommand{\bqa}{\begin{eqnarray*}}
\newcommand{\eqa}{\end{eqnarray*}}

\begin{document}
\maketitle


\begin{abstract}
It is natural to expect the following loosely stated approximation principle to hold: a numerical
approximation solution should be in some sense as smooth as its target exact
solution in order to have optimal convergence. For piecewise polynomials,
 that means we have to at least maintain
 numerical smoothness in the
interiors as well as across the interfaces of cells or elements.
In this paper we give clear definitions of numerical smoothness that address the across-interface smoothness in terms of {\it scaled} jumps in derivatives \cite{SUN} and the interior numerical smoothness in terms of differences in derivative values. Furthermore, we prove rigorously that the principle can be simply stated as {\it numerical smoothness is necessary for optimal order convergence.}
It is valid on quasi-uniform meshes by triangles and quadrilaterals in two dimensions and by tetrahedrons and hexahedrons
in three dimensions. With
this validation we can justify, among other things, incorporation of this
principle in creating adaptive numerical approximation for the solution of
PDEs or ODEs, especially in designing proper smoothness indicators or detecting potential non-convergence and instability.
\end{abstract}



\section{Introduction}
Consider the problem of approximating a function $u$ defined on a domain in ${\mathbb R}^n$ by a sequence of numerical
solutions $\{u_h\}$. The target function $u$ may be an exact solution of a second or higher order partial or ordinary differential equation, and the sequence may be piecewise polynomials from a discontinuous Galerkin method \cite{HR} or reconstructed polynomials $u^R$ in an intermediate phase \cite{LEVEQUE}, and even post-processed finite element solutions to achieve superconvergence \cite{LW}. Although we had discontinuous Galerkin numerical solutions in mind, the source of the problem is not important for our purpose here, as it only puts the degree of smoothness of $u$ in perspective. Now suppose that $u$ is in $W^{p+1}_s(\Omega)$ (standard notation for Sobolev spaces here, supindex for the order
 of derivative and subindex for the $L^s$ based space). It is natural to expect that the approximation solutions should be as smooth (in some sense) in order to achieve optimal convergence rate. The purpose of this paper is to give clear and rigorous results on this simple minded principle.

 Sun \cite{SUN} showed in one dimension if the mesh is uniform and the function $u$ has $p+1$
 weak derivatives in $L^s,s=1,2,\infty$, then a necessary condition can be formulated. In particular in the $s=\infty$ case,
 the jumps of the $k^{th}$ derivatives, (across a mesh point) of the approximation piecewise polynomial of degree $p$ must be less than or equal to
  ${\mathcal O} (h^{p+1-k}), 0\leq k\leq p$. This one dimensional result is perhaps not surprising, once one realizes the interpolation error behaves in a similar way: taking a derivative, one loses a power of $h$, assuming $u\in W^{p+1}_\infty$. In the appendix of this paper, the assertion is actually proved by comparing the derivatives of $u$, its continuous piecewise Lagrange interpolant $u^I$, and $u^R$ at a mesh point.  This short proof can even be carried over to higher dimensions. Unfortunately, it cannot be extended to higher dimensions when $s=1,2$ due to the restriction on continuity imposed by the Sobolev imbedding theorem (See Remark \ref{A.1} in the appendix for other reasons). Since now one starts with a function $u\in W^{p+1}_s,s=1,2$, there are always some $k$ and up for which the $k$th derivative of $u$ at a point of interest is
 not well defined. On the other hand, in hindsight an idea (Lemma \ref{crucial} below) in the much lengthier and originally unfavored proof in \cite{SUN} for one dimension can be distilled and generalized to prove the two and three dimensional versions of the same principle.

  While Sun\,{\it et\,al.} \cite{SUN2,SR} have successfully applied it to the analysis of numerical methods for one dimensional nonlinear conservation laws, it is quite
  clear that this principle has a very broad scope of applications such as safeguarding divergence or negating optimal order convergence in designing new methods, let alone in creating smoothness indicators \cite{SUN2,SR} in an adaptive algorithm, and so on. Being motivated by its application potential
  in higher dimensions, in this paper we generalize the concept of numerical smoothness of a piecewise polynomial in \cite{SUN} to higher dimensions and show that in order for the convergence of $u_h$ to $u$ to have optimal order $p$ in $W^{p+1}_s$, $u_h$ must have $W^{p+1}_s,s=1,2,\infty$ numerical smoothness, provided that the domain can be meshed by quasi-uniform subdivisions into triangles or quadrilaterals in 2-D and tetrahedrons or hexahedrons (cubes) in 3-D. We accounted for both interior and across interface numerical smoothness. In $\S$ 3, we formerly define the across-interface numerical smoothness in Definition \ref{smooth_def}, which is well motivated by the theorems in $\S$ 2 and also define interior numerical smoothness in Definition \ref{smooth_def1}. The main result
 that states optimal order convergence implies numerical smoothness is proved in Theorems \ref{smoothness} and \ref{smoothness1}. This section is
 written in such a way that the reader can go read it directly after the introduction section.

   The organization
 of rest of this paper is as follows. In $\S$ 2, we first derive basic error estimates without imposing conditions on meshes other than the shape regularity. The main theorem  is Theorem \ref{main2}, now under the quasi-uniform condition on the mesh.
 Finally, in $\S$ 4 we give a short proof of the one dimensional version of Theorem \ref{smoothness} and explain why it cannot be extended to higher dimensions.

\section{Basic Estimates for Numerical Smoothness}\label{section2}
Let $\alpha=(\alpha_1,\alpha_2,\cdots,\alpha_n), \alpha_i\geq
0,1\leq i\leq n$ be a multi-index and
$|\alpha|=\sum_i^n\alpha_i.$ Some of the theorems in this section
have their one dimensional counterparts in \cite{SUN}. We are especially inspired by
the central use of Lemma 2.2 in \cite{SUN}. The next lemma is its higher dimensional version, which will be used after
 a scaling argument back to the master element of unity size. At a certain point $x\in {\mathbb R}^n$ of interest,
e.g., a mesh nodal point, a center of a simplex (edge or face), to measure the smoothness of a mesh function $u_h$, we will be examining
 all the jumps $\llbracket\partial^\alpha u_h\rrbracket_x,|\alpha|=k$ in the partial
derivatives of order $k$ for $0\leq k\leq p$. In this perspective, we now state and prove the next lemma.  Denote by ${\mathbb P}_p$ the space of polynomials of total degree at most $p$.

\begin{lemma}\label{crucial}
Let $\Delta=(\Delta_0,\Delta_1,\cdots,\Delta_p)$, where each $\Delta_k$ is a vector of a certain length (e.g., it has as many components as the number of partial derivatives of order $k$). Let $\hat \Omega_{\pm}$
be two open set in ${\mathbb R}^n$.  Define
\[Q(\Delta)=\min_{\hat v\in \scp}\left(
\left\|\hat v +\frac
12\sum_{\alpha\in {\mathcal I}}\frac{\Delta_\alpha}{\alpha!}\xi^\alpha\right\|^2_{L^2(\hat\Omega_-)}
+ \left\|\hat v -\frac
12\sum_{\alpha\in {\mathcal I}}\frac{\Delta_\alpha}{\alpha!}\xi^\alpha\right\|^2_{L^2(\hat \Omega_+)}
\right),\]
where the minimum is taken over $\scp={\mathbb P}_p$ in $\xi$ and the index set
\begin{equation}\label{index_set}
{\mathcal I}={\mathcal I}_\ps:=\{\alpha:|\alpha|=k,\,\,0\leq k\leq p\}
\end{equation}
Then $Q(\Delta)$ is a positive definite quadratic form in $\Delta$, and
 there exists a constant $C_p>0$ such that
\begin{equation}\label{cs}
Q(\Delta)\ge C_p \|\Delta\|^2=C_p\sum_{i=0}^p\|\Delta_i\|^2,
\end{equation}
where $\|\Delta_i\|$ is the spectral norm of vector $\Delta_i$.
\end{lemma}
\begin{proof} The $n=1$ version is in \cite{SUN}.
Simply notice that the minimizer $\sum_\alpha V_\alpha\xi^\alpha$ is such that
each $V_\alpha$ is a linear combination of $\Delta_\beta$'s. Non-degeneracy
comes from the fact that $V_\alpha+\frac 12\frac {\Delta_\alpha}{\alpha!}=0$ and
$V_\alpha-\frac 12\frac {\Delta_\alpha}{\alpha!}=0$ implies $\Delta_\alpha=0$.
\end{proof}

{\it We will write $Q(\Delta)$ as
$Q(\Delta_0,\Delta_1,\cdots,\Delta_p)$ if this longer notation can
be accommodated in display.}

\begin{theorem}\label{thm2.2}
Suppose that $u\in H^{p+1}(\Omega)$, $\Omega=(a,b)$
and $u^R$ is a piecewise polynomial of degree $\leq p$ with respect to a subdivision
$\{{\mathcal T}_h\}$ of subintervals $\{I_i=(x_i,x_{i+1})\}_{i=0}^N$. Let $h=\max_i|I_i|$
and $h_{\min}=\min_i|I_i|$.
 Then there exists a positive constant $C_1$
independent of $h,u$ and $u^R$ such that
\begin{equation}
 \norm u-u^R\norm_{L^2(\Omega)}\geq C_1h^{p+1}\left(
\sqrt{\sum_{i=1}^{N}h^n_{\min}\norm \tilde
D_i\norm^2}-|u|_{H^{p+1}(\Omega)}\right),
\end{equation}
 where $n=1$ and the components of
$\tilde D_i$ are
\begin{equation}\label{jump}
\tilde
D_i^{\alpha}=J_i^{(\alpha)}/(h^{p+1}h^{-|\alpha|}_{\min}),\quad
J^{(\alpha)}_i=\llbracket\partial^\alpha u^R\rrbracket_{x_i},\quad
|\alpha|=k,\,\, 0\leq k\leq p.
\end{equation}
\end{theorem}
\begin{remark} In the above statement, we used higher dimension notation for 1-D case so that its extension to
higher dimensions is more clear and straightforward. We will do the same in the proof below.
\end{remark}

\begin{proof}  Let ${\mathcal T}_h$ be a partition $x_0=a<x_1<x_2<\cdots<x_{N}<b=x_{N+1}$ on $(a,b)$. Define a new partition
by adding the midpoint $x_{i+1/2}$, center of each element $(x_{i},x_{i+1})$, to the old partition. We call $(x_{i-1/2},x_i)\cup (x_i,x_{i+1/2})\cup \{x_i\}$ the covolume associated with $x_i$. The closure of all covolumes associated with interior nodes, together with the two boundary covolumes $(x_0,x_{1/2})$
and $(x_{N+1/2},x_{N+1})$ form a dual mesh $\{{\mathcal T_h^*\}}$ to ${\mathcal T}_h$ , which we call dual covolume mesh.
Let $\ps^*_h$ be the space of piecewise polynomials of degree $\le p$ with respect to this dual mesh. Since $u\in H^{p+1}(\Omega)$, there exists
a $u^I\in \ps^*_h$ such that
\begin{equation}\label{covolume}
   \norm u-u^I\norm_{L^2(\Omega)}\le C_2h^{p+1}|u|_{H^{p+1}(\Omega)}.
\end{equation}
On the other hand, by the triangle inequality
\begin{eqnarray*}
 \norm u-u^R\norm_{L^2(\Omega)} \geq\norm u^I-u^R\norm_{L^2(\Omega)}- \norm u-u^I\norm_{L^2(\Omega)}\\
 \geq \sqrt{\sum_{i=1}^{N} \norm u^I-u^R\norm^2_{L^2(\Omega_i^*)}}-C_2h^{p+1}|u|_{
 H^{p+1}(\Omega)},
\end{eqnarray*}
where $\Omega_i^*$ is a smaller subset of the covolume associated with $x_i$. Now let's choose $\Omega^*_i$.
At each $x_i, 1\leq i\leq N$, we take a 1-D ball $\Omega^*_i=\{|x-x_i|\leq \delta\}$ so small so that it is in $(x_{i-1},x_{i+1})$. It is essential for the later use of Lemma \ref{crucial} that $\delta$ works for all $i$. We take $\delta=\frac 14 h_{\min}$.

 Let $\{q\}$ denote the average
of $q^+$ and $q^-$ and let $\llbracket q \rrbracket=q^+-q^-$ denote the
jump. Then it is trivial that
\begin{equation}\label{ja}
q^+-\{q\}=\frac 12\llbracket q \rrbracket\quad \mbox{  and  } q^--\{q\}=-\frac 12\llbracket q \rrbracket.
\end{equation}
Applying this with $q=q^k_i=\frac{d^ku^R}{dx^k}(x_i)$, the
discontinuous $k$th derivative of $u^R$ at $x_i$, and letting
$w(x)=\sum_{k=0}^p\frac{\{q_i^k\}}{k!}(x-x_i)^k$, we have with
$J_i^{(k)}:=\llbracket\frac{d^ku^R}{dx^k}\rrbracket_{x_i}$ that
\[ u^R-w=\frac 12\sum_{k=0}^p\frac{J_i^{(k)}}{k!}(x-x_i)^k,\quad \forall x\in\Omega_{i,+}^*:=(x_i,x_i+\delta)\]
and
\[ u^R-w=-\frac 12\sum_{k=0}^p \frac{J_i^{(k)}}{k!}(x-x_i)^k, \quad \forall  x\in\Omega_{i,-}^*:=(x_i-\delta,x_i).\]
Now using the change of variable $\xi=(x-x_i)/h_{\min}$ below, we have
\begin{align}
\norm\,& u^I-u^R\norm^2_{ L^2( \Omega_i^*) }\geq
 \min_{v\in \ps_p
}\norm v-u^R\norm^2_{ L^2( \Omega_i^*) }\label{sub}\\
&=\min_{v\in \ps_p}\norm v-(u^R-w)||^2_{ L^2( \Omega_i^*)}\notag\\
 &=\min_{v\in \ps_p}\left( \left\| v +\frac
12\sum_{k=0}^p\frac{J_i^{(k)}}{k!}(x-x_i)^k\right\|^2_{L^2(\Omega_{i,-}^*)} +
\left\| v -\frac
12\sum_{k=0}^p\frac{J_i^{(k)}}{k!}(x-x_i)^k\right\|^2_{L^2(\Omega_{i,+}^*)}
\right)\notag\\
&=h^n_{\min}\min_{\hat v\in \hat \ps_p}\left( \left\|\hat v +\frac
12\sum_{k=0}^p\frac{J_i^{(k)}}{k!}(\xi h_{\min})^k\right\|^2_{L^2(\hat\Omega_-)} +
\left\|\hat v -\frac
12\sum_{k=0}^p\frac{J_i^{(k)}}{k!}(\xi h_{\min})^k\right\|^2_{L^2(\hat\Omega_+)}
\right),\notag\\
&\qquad \qquad \qquad \mbox{ where  } \hat \Omega_-=(-\frac 1 4,0),\, \hat\Omega_+=(0,\frac 14)\notag\\
&=(h^n_{\min} h^{2p+2})\,
\min_{\hat v\in \hat\ps_p}\left( \left\|\hat v +\frac
12\sum_{k=0}^p\frac{\tilde D_i^{(k)}}{k!}\xi^k\right\|^2_{L^2(\hat\Omega_-)} + \left\|\hat v -\frac 12\sum_{k=0}^p\frac{\tilde D_i^{(k)}}{k!}\xi^k\right\|^2_{L^2(\hat\Omega_+)}
\right),\label{extract}      \\
&\qquad (n=1)\notag\\
&=h^n_{\min} h^{2p+2}Q(\tilde D_i^0,\tilde D_i^1,\cdots, \tilde
D_i^p).\notag
\end{align}

(Note that the minimization--range's change to ${\mathbb P}_p$ in (\ref{sub}) was possible due to the fact that $u^I\in {\mathbb P}^*_h$ is a single piece of polynomial on the covolume.)

Now invoking (\ref{cs}) on
\begin{eqnarray}\label{2.4}
\norm u^I-u^R\norm^2_{ L^2( \Omega_i^*) }\geq h^n_{\min} h^{2p+2}Q(\tilde D_i^0,\tilde D_i^1,\cdots, \tilde D_i^p).
\end{eqnarray}
completes the proof.
\end{proof}
 Next we move to the 2-D case. For a given $\{\mathcal T_h\}$, we will
use a dual mesh $\{{\mathcal T_h}^*\}$ consisting of covolumes as in Chou and Vassilevski \cite{CV}. Covolumes are obtained by adding a point to an old element in $\{\mathcal T_h\}$
and connect it with the vertices of the element. The covolumes around an edge or associated with the midpoint of the edge is shown in Figure  \ref{fig1}
 (see diamond $STVU$ associated with $x_i$). Note that a covolume is obtained by connecting vertices with an added point, which could be a circumcenter or a barycenter for triangular grids and intersection of diagonals for quadrilateral grids.
We can now summarize the most important ingredients of the 1-D proof.
\begin{itemize}
\item[(i)] {\it Ensure that the optimal order estimate (\ref{covolume}) in the $L^2$ norm holds on the dual covolume mesh.}

In contrast, unlike in 1-D the dual mesh shape condition comes into play as well. In higher dimensions,
the construction of the dual covolume mesh in a symmetric and smooth way causes the shape regularity of the primary mesh to be inherited. Furthermore the dual mesh of a triangular or quadrilateral mesh in the covolume construction is almost a quadrilateral mesh (boundary covolumes are triangles), but we will not need the boundary covolumes in our analysis below. In 3-D a covolume is the union of two tetrahedrons when the primary mesh is tetrahedral. Here the approximation
order is usually achieved by the existence of a good local $L^2$ projection-type interpolation operator. See Girault and Raviart \cite[pp. 101-109]{GR} for such operators.
    \item[(ii)] {\it Ensure the choice of radius $\delta$ is such that $\delta/h_{min}$ is a constant, independent of $h$.}

    This scales $\Omega_i^*$ to unit size so that Lemma \ref{crucial} can be used to extract a lower bound with the constant independent of $h$. Of course in this step
        the $\Omega_i^*$s are automatically non-overlapping due to the covolume construction.
    \item[ (iii)] {\it The order of approximation on the dual mesh limits the extracting power.}

    Note that what power of $h$ in (\ref{extract}) to extract is determined by how well the approximation on the dual mesh can be done.
    The optimal case is $h^{p+1}$.
\end{itemize}
 With this in mind, we now prove the corresponding theorem for triangular and quadrilateral meshes. We will denote by
${\mathbb P}_p^h$  for the piecewise polynomial space associated with ${\mathbb P}_p$. The regularity of a family of quadrilateral
subdivisions is defined as follows \cite[p. 104]{GR}. Let $Q$ be a quadrilateral with four vertices $\bx_i$ and
denote by $S_i$ the subtriangle of $Q$ with vertices
$\bx_{i-1},\bx_i$ and $\bx_{i+1}$ ( $\bx_0=\bx_4$). Let $h_Q$ be
the diameter of $Q$ and $\rho_Q= 2 \min_{1\le i\le 4}\{\mbox{
diameter of circle inscribed in } S_i\}$.
A family of quadrilateral partitions
 $\{ {\mathcal Q}_h \}$ is said to be regular if there exists a positive constant $\sigma$, independent of
$h$, such that
\begin{equation}\label{regular}
 \frac{h_Q}{\rho_Q}\le \sigma, \,\quad \forall Q\in
{\mathcal Q}_h, \forall {\mathcal Q}_h\in {\mathcal Q}.
 \end{equation}
 There are equivalent definitions \cite{CH}.


\begin{figure}[!htb]
\centering
\includegraphics[scale=0.55]{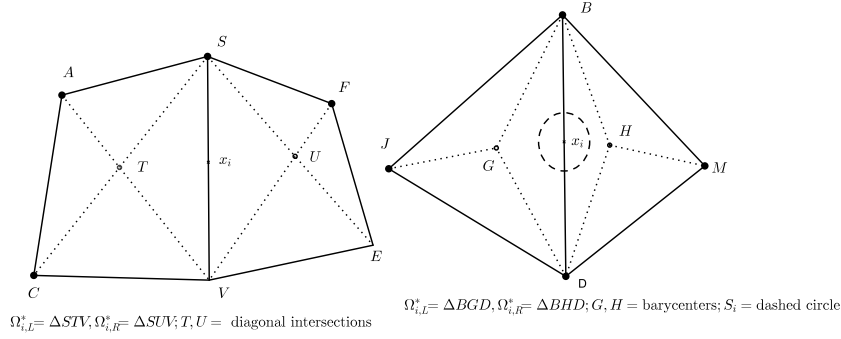}
\caption{Quadrilateral mesh with its covolumes ($STVU$) and Triangular mesh with its covolumes $BGDH$ and selected disk}
\hspace{-2cm}
\label{fig1}
\end{figure}


\begin{theorem}\label{thm2.3}
Let $u\in H^{p+1}(\Omega)$, $\Omega\subset {\mathbb R}^2$. In addition, let the following assumption hold.
\begin{itemize}
\item[H1.]  $u^R\in {\mathbb P}_p^h $ on a family of regular subdivisions $\{{\mathcal T}_h\}$ by triangles or by quadrilaterals,
 i.e., $u^R\in {\mathbb P}_p$ on each element.
 \end{itemize}
 Then there exists a positive constant $C_1$
independent of $h,u$ and $u^R$ such that
\[ \norm u-u^R\norm_{L^2(\Omega)}\geq C_1h^{p+1}\left(
\sqrt{\sum_{i=1}^{N^\circ_e}h^n_{\min}\norm \tilde
D_i\norm^2}-|u|_{H^{p+1}{(\Omega)}}\right), \] where $n=2$ and the
components of $\tilde D_i$ are
\begin{equation}\label{jump2}
\tilde
D_i^{\alpha}=\frac{J_i^{(\alpha)}}{h^{p+1}h^{-|\alpha|}_{\min}},\quad
J^{(\alpha)}_i=\llbracket\partial^\alpha u^R\rrbracket_{x_i},\,|\alpha|=k,\,\, 0\leq k\leq p.
\end{equation}
Here $N^\circ_e$ is the number of interior edges and $h_{\min}$ is
the least edge length.
\end{theorem}

\begin{proof} We will proceed as in 1-D case, pointing out the difference along the way.\\
{\it Case 1: Triangular mesh.}\\
 For each ${\mathcal T}_h$ we define
a dual mesh ${\mathcal T}^*_h$ formed as follows. With reference of Figure \ref{fig1}, in each element we connect the vertices (e.g., $B,M,D$) with a newly added point, which in this case
is the barycenter (e.g., $H$), to create three new triangles.
 The two half-covolumes (e.g. $BGD$, $BHD$) form a single covolume $BGDH$ for the common edge $BD$.  All covolumes form the dual mesh ${\mathcal T}^*_h$.
 We can find a $u^I\in {\mathbb P}^*_h$ so that (\ref{covolume}) holds under no regularity conditions on the dual mesh by quadrilaterals. The $u^I$ is the {\em local} $L^2$ projection and estimate (\ref{covolume}) can be found in \cite[p. 108]{GR}.

   Let $x_i$ be the midpoint of an interior edge $e_i$ common to two half-covolumes $\Omega_{i,L}^*,\Omega_{i,R}^*$. In Figure \ref{fig1}
   , $\Omega_{i,L}^*=\triangle BGD,\Omega_{i,R}^*=\triangle BHD$. Let us take $\Omega_i^*=S_i$ to be an open disk with center $x_i$ and a radius $\delta$ small enough so that $S_i$ is fully
 contained in the interior of $\bar \Omega_{i,L}^*\cup\bar \Omega_{i,R}^*$.
 The radius however has to work for all midpoints $x_i$ on interior edges. It is well known that shape regularity is equivalent to
 the minimal angle condition, and consequently there is a constant $\theta_0$ such that
  all interior angles $\theta\ge \theta_0$ for all $h$. Without loss of generality,
   suppose the distance from $x_i$ to the boundary of $\bar \Omega^*_{i,L}\cup\bar \Omega_{i_R}^*$ is
    attained by $|\overline{x_iF}|$, where the foot $F$ is on $\overline{BH}$. Then
 \begin{eqnarray*}
 |\overline{x_iF}|&=|\overline{Bx_i}|\sin \frac 1 2\angle MBD\geq
 |\overline{Bx_i}|\sin\frac{\theta_0}{2}\geq \frac 1 2h_{\min}\sin\frac{\theta_0}{2},
    \end{eqnarray*}
    where we have used the fact that the sine function is increasing on $[0,\frac \pi 2]$ and that $BH$ is a bi-angle line.
Thus it suffices to take $\delta=\frac 14 h_{\min}\sin\frac{\theta_0}{2}$ as the common radius.  The rest of proof is just like 1-D case.
For example, now
\begin{equation}\label{sam1} u^R-w=\frac 12\sum_{k=0}^p\sum_{|\alpha|=k}\frac{J_i^{(\alpha)}}{\alpha!}(x-x_i)^\alpha,\quad \forall x\in\Omega_{i,+}^*,
\end{equation}
and
\begin{equation}\label{sam2}
 u^R-w=-\frac 12\sum_{k=0}^p \sum_{|\alpha|=k}\frac{J_i^{(\alpha)}}{\alpha!}(x-x_i)^\alpha,\quad \forall x\in\Omega_{i,-}^*.
 \end{equation}
 Furthermore,  from the validation of
 \begin{align*}
\norm\, u^I-u^R\norm^2_{ L^2( \Omega_i^*) }&\geq
 \min_{v\in \ps_p
}\norm v-u^R\norm^2_{ L^2( \Omega_i^*) }\notag\\
&=\min_{v\in \ps_p}\norm v-(u^R-w)||^2_{ L^2( \Omega_i^*)}\notag
\end{align*}
and  use of $\xi=(x-x_i)/h_{\min}$  and Lemma \ref{crucial}, we can derive as in the corresponding 1-D case that
\begin{align*}
\norm\, u^I-u^R\norm^2_{ L^2( \Omega_i^*) }\geq h^n_{\min} h^{2p+2}Q(\tilde D_i^0,\tilde D_i^1,\cdots, \tilde
D_i^p),\quad n=2.
\end{align*}
{\it Case 2: Quadrilateral mesh under assumption $H1$.} \\
First, optimal order estimate (\ref{covolume}) holds with the local $L^2$ projection as before.
 Equations (\ref{sam1})-(\ref{sam2}) are valid as well. So the
only concern is the choice of $\delta$. Since the local geometry in the left figure of Figure  \ref{fig1} is the
same as in the triangular case, we would still need the minimal
angle condition. However, it is shown in Theorem 4.1 of  Chou and He \cite{CH},
regularity of the quadrilateral meshes defined in \cite{GR} implies the minimal angle condition:
all interior angles of quadrilaterals and the interior angles of the subtriangles $S_i$ in (\ref{regular}) are bounded below, although
the converse is not true. On the other hand, suppose the shortest
distance from $x_i$ to the covolume edges is attained by $F$ on
$\overline{SU}$. It should be clear that $\angle VSU$ can never
exceed $90$ degrees as well. Thus, we can still take the same
$\delta$.

 This completes the proof.
\end{proof}

\begin{figure}[!htb]
\centering
\includegraphics[scale=0.55]{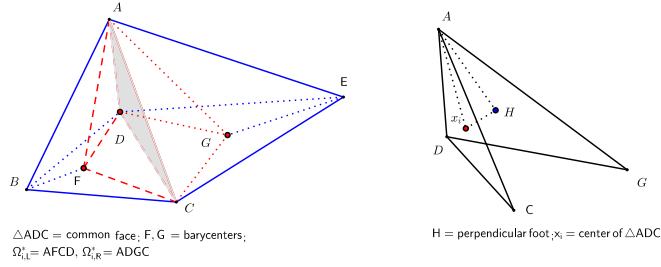}
\caption{Tetrahedral mesh with a typical covolume (tetrahedron $ADCG$ union tetrahedron $ADCF$) and local configuration around a center point in common face}
\label{fig2}
\end{figure}

We now prove the 3-D version of the previous theorem. In 3-D,
regularity of tetrahedral subdivision is still defined in terms of
the uniform boundedness of $h_K/\rho_K$, the ratio  of the maximum
diameter $h_K$ to $\rho_K$, the diameter of the inscribed sphere. It is shown
in Brandts {\it et\, al.} \cite{BKK} that this condition is equivalent to the minimal angle
condition, but now two types of angles are involved, angles between edges and between faces. We state this
equivalence in the next lemma.
\begin{lemma}\label{min}
Let $\{{\mathcal T}_h\}_{h>0}$ be a family of tetrahedral subdivisions for a domain in ${\mathbb R}^3$. Then $\{{\mathcal T}_h\}_{h>0}$ being regular is
equivalent to the minimal angle condition:
 there exists a constant $\beta_0>0$ such that
for any partition ${\mathcal T}_h$, any tetrahedron $T\in {\mathcal T}_h$, and any dihedral angle (angle between faces) or
solid angle (angle between edges) $\beta$ of $T$, we have
\begin{equation}
\beta\geq \beta_0.
\end{equation}
\end{lemma}

\begin{theorem}\label{thm3D}
Let $u\in H^{p+1}(\Omega)$, $\Omega\subset {\mathbb R}^3$. In addition, let the following assumption hold.
\begin{itemize}\label{2H}
\item[A1.]  $u^R\in {\mathbb P}_p^h $ on a family of regular subdivisions $\{{\mathcal T}_h\}$ by tetrahedrons or hexahedrons.
 i.e., $u^R\in {\mathbb P}_p$ on each element.
 \end{itemize}
 Then there exists a positive constant $C_1$
independent of $h,u$ and $u^R$ such that
\[ \norm u-u^R\norm_{L^2(\Omega)}\geq C_1h^{p+1}\left(
\sqrt{\sum_{i=1}^{N^\circ_f}h^n_{\min}\norm \tilde
D_i\norm^2}-|u|_{H^{p+1}(\Omega)}\right), \] where $n=3$ and the
components of $\tilde D_i$ are
\begin{equation}\label{jump22}
\tilde
D_i^{\alpha}=\frac{J_i^{(\alpha)}}{h^{p+1}h^{-|\alpha|}_{\min}},\quad
J^{(\alpha)}_i=\llbracket\partial^\alpha u^R\rrbracket_{x_i},\quad
|\alpha|=k,\,\, 0\leq k\leq p.
\end{equation}
Here $N^\circ_f$ is the number of interior faces and $h_{\min}$ is
the least edge length.
\end{theorem}
\begin{proof} We consider the tetrahedral mesh case first.
 In Figure \,\ref{fig2}, $\triangle ADC$ is a typical common tetrahedral element interface with the accompanying
 half-covolumes $\Omega_{i,L}^*=AFCD$ and $\Omega_{i,R}^*=AGCD$. The barycenter of $\triangle ADC$ is ${x_i}$ (not labeled to avoid clustering).
 We need to choose the common radius $\delta$ of the sphere centered at $x_i$ that works for all $i$.  Assume that the shortest distance from $x_i$ to the boundary
 of covolumes is attained by the plane containing $\triangle ADG$. In the right side of Figure  2, a blowup of the situation is
 shown, the perpendicular foot from $x_i$ is $H$. Let $A_M$ be the midpoint of $CD$ and let $A_{\perp}$ be the foot on $CD$ of the perpendicular
 from $A$, then
\begin{equation}\label{M}
  |\,\overline{AA_M}\,|\geq |\,\overline{AA_\perp}\,|=|\,\overline {AC}\,|\sin\angle
 ACD \geq h_{\min}\sin \beta_0,
 \end{equation}
 where $\beta_0$ is the common lower bound for angles between edges in the minimal angle
 condition.

  Now let $K$ be the perpendicular foot on $\overline{AD}$ from $x_i$ (see Figure  \ref{fig3} for the blowup version) and let $\theta_s=\angle HKx_i$. In general $\theta_s$ is not
  equal to $\theta_f$, the (dihedral) angle between planes determined by $ADC$ and $ADG$ unless $\overline{HK}\perp \overline{AD}$. Since $\overline{x_iH}$ is normal to plane $ADG$
  and $\overline{x_iK}$ is normal to $\overline{AD}$, the two angles are expected to be related. This is indeed the case. In fact,

 \begin{equation}\label{ang}
 \cos \theta_f=\frac 1 3\cos\theta_s,
 \end{equation}
 which will be proved in Lemma \ref{angle_relation} below.
Concentrating on $ \triangle Ax_iD$, we see that
\begin{eqnarray*}
|\,\overline{Kx_i}\,|&=&|\,\overline{Ax_i}\,|\sin\angle x_iAD=|\,\overline{Ax_i}\,|\sin \frac 12\angle CAD\\
&=&\frac 23 |\,\overline{AA_M}\,|\sin \frac 12\angle CAD\\
&\geq& \frac 23 h_{\min}\sin\beta_0\sin\frac {\beta_0}{2},
\end{eqnarray*}
where we used (\ref{M}) in the last inequality.

With reference to Figure \ref{fig4} (local blowup version of the left figure in Figure \ref{fig1}), let $\theta_{ext}=\angle(ADC,ADE)$, the angle between the planes $ADC$ and$ADE$. Then there exists a constant $\gamma>1$ independent of $h$ such that
\begin{equation}\label{compare}
  \cos\theta_f\leq \gamma \cos\theta_{ext},
  \end{equation}
  whose proof is given in Lemma \ref{angle_relation2} below.

Letting $g(\theta_f):=\sin\left(\cos^{-1}(3\cos\theta_f)\right)$ (an increasing function on $[0,\pi/2]$), and using (\ref{ang})--(\ref{compare}), we have
\begin{eqnarray*}
  \delta&=&|\,\overline{x_iH}\,|=|\,\overline{Kx_i}\,|\sin\theta_s\\
  &\geq& \frac 2 3 h_{\min}\sin\frac {\beta_0}{2}\sin\beta_0\,g(\theta_f)\\
  &\geq& \frac 2 3 h_{\min}\sin\frac {\beta_0}{2}\sin\beta_0\,g(\cos^{-1}\gamma\cos\theta_{ext})\\
  &\geq&\frac 2 3 h_{\min}\sin\frac {\beta_0}{2}\sin\beta_0\,g(\cos^{-1}\gamma \cos\beta_0),
\end{eqnarray*}
where we have used the minimal angle condition.
Thus we can take $\delta$ to be half of the last expression in the above inequality. The rest of the proof is as before.

As for hexahedral case, its local geometry around $x_i$ is similar to the tetrahedral case. Similar comments made for quadrilateral case
apply here too.  This completes the proof.
\end{proof}
\begin{figure}[!htb]
\centering
\includegraphics[scale=0.55]{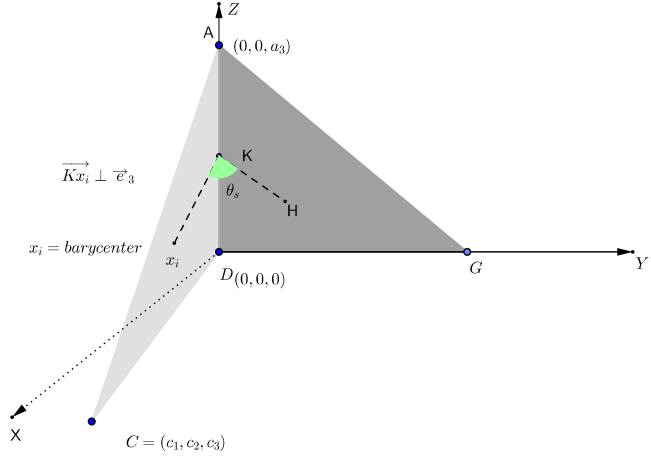}
\caption{ If $\theta_f=$ dihedral angle between $YZ-$plane and $\triangle ADC$, $\theta_s=\angle HKx_i$, then
$\cos \theta_f=1/3 \cos \theta_s$}
\label{fig3}
\end{figure}

\begin{figure}[!htb]
\centering
\includegraphics[scale=0.45]{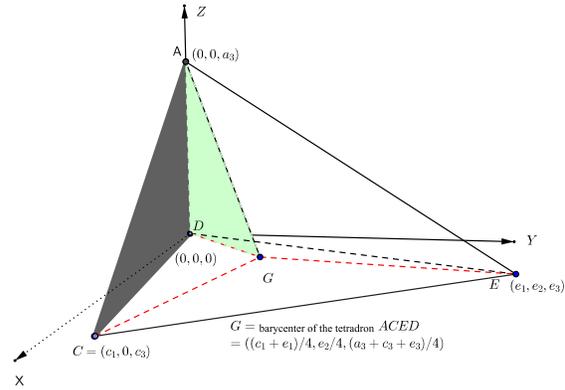}
\caption{ If $\theta_f=\angle(ADC,ADG)$ and  $\theta_{ext}=\angle(ADC,ADE)$, then
$\theta_f<\theta_{ext}$}
\label{fig4}
\end{figure}

\begin{lemma}\label{angle_relation}
With reference to Figure  \ref{fig3}, let $AD$ be the intersection of two triangles $ADC$ and $ADG$ in ${\mathbb R}^3$ and denote by
 ${\theta_f}$  $\angle(ADC,ADG)$, the angle between them (angle between their normals).
 Let $x=x_i$ be the barycenter of $\triangle ADC$,
 $K\in \overline{AD}$, $\overline{xK}\perp \overline{AD}$ and $H\in \triangle ADG$ such that
  $\overline{xH}\perp \triangle ADG$. Let $\theta_s=\angle HKx$. Then
\[\cos {\theta_f}=\frac 1 3\cos\theta_s.\]

\end{lemma}
\begin{proof}
By the properties of $x$, $H$ and $K$, one derives
\[x=(\frac {c_1}{3},\frac {c_2}{3},\frac{c_3+a_3}{2}),\quad H=(0,\frac{c_2}{3},\frac{c_3+1}{3}),\quad K=(0,0,\frac {a_3+c_3}{3}).\]
Let $n_1=(1,0,0)$ and $n_2=\overrightarrow{xK}\times\overrightarrow{xC}$ normalized. Using $\overrightarrow{KH}$ and $\overrightarrow{Kx}$ to compute
$\cos\theta_s$ and using $\cos\theta_f=n_1\cdot n_2$, after some computation we get $\cos\theta_f=1/3\cos\theta_s$.
\end{proof}

\begin{lemma}\label{angle_relation2}
 With reference to Figure  \ref{fig4}, let $AD$ be the intersection of two triangles $ADC$ and $ADG$ in ${\mathbb R}^3$ and let
 \[{\theta_f}=\angle(ADC,ADG), \qquad \theta_{ext}=\angle (ADC, ADE).\]
Then there exists a constant $\gamma >0$ independent of $h$ such that
\begin{equation}\label{angle2}
  \cos\theta_{ext}<\cos\theta_f\leq \gamma \cos\theta_{ext}.
  \end{equation}
\end{lemma}
\begin{proof} The lemma is {\it stated and proved} in the context of the main theorem and hence Figure \ref{fig4} configuration is assumed.
Since $G$ is the barycenter, it is the average of four vertices of $ADCE$,
\[G=(\frac{c_1+e_1}{4},\frac{e_2}{4},\frac{a_3+c_3+e_3}{4}),\]
and three normal vectors (some not normalized) to the planes from left to right in Figure \ref{fig4} are
\[n_1=-\overrightarrow{j},\quad n_2=-e_2\overrightarrow{i}+(c_1+e_1)\overrightarrow{j},\quad n_3=e_2\overrightarrow{i}-e_1\overrightarrow{j}.\]
Thus \[
\cos\theta_f=\frac{c_1+e_1}{ \sqrt{(c_1+e_1)^2+e_2^2} },\qquad \cos \theta_{ext}=\frac{e_1}{\sqrt{e_1^2+e_2^2}}.
\]
In our figure,we have $c_1>0$ and $e_1>0$, which is the case for the shortest distance case.
Note that $f(c)=\frac{c+e_1}{ \sqrt{(c+e_1)^2+e_2^2} }$ is increasing on $[0,\infty]$ and so  $f(c_1)\geq f(0)$ implies $\cos\theta_f\geq \cos\theta_{ext}$. On the other hand, $c_1$ is an edge size and so we can assume it is bounded by $1$. Thus
 we use $f(c_1)\leq f(1)$. In turn since $f(1)$ is comparable to $\frac{e_1}{\sqrt{e_1^2+e_2^2}}$, we see that there exists a $\gamma$ independent of $h$
such that (\ref{angle2}) holds.
\end{proof}

\begin{theorem}\label{thm2.4}
 Suppose that
 \[\text{H0}:\hfill u \mbox{ is in } W_1^{p+1}(\Omega),\Omega\subset {\mathbb R}^n,\]
 where $n=2$ and that assumption H1 of Theorem \ref{thm2.3} holds. Then there are constants $C_1>0$ independent of $h,u$ and $u^R$ such
that
\begin{equation}\label{above}
 \norm u-u^R\norm_{L^1(\Omega)}\geq C_1
h^{p+1} \left[ \sum_{1\leq i\leq N^\circ_e} h^n_{\min}\norm\tilde
D_i\norm-|u|_{W^{p+1}_1(\Omega)}\right],
\end{equation}
where $\tilde D_i$ has components
\begin{equation}\notag
\tilde D_i^\alpha=J_i^{(\alpha)}/(h^{p+1}h^{-|\alpha|}_{\min}),\quad J^{(\alpha)}_i=\llbracket\partial^\alpha u^R\rrbracket_{x_i},\,\,
|\alpha|=k,\,\, 0\leq k\leq p.
\end{equation}

Suppose assumptions H0 and A1 of Theorem \ref{thm3D}  holds for $n=3$. Then (\ref{above}) holds with $N_e^\circ$ replaced by $N_f^\circ$.
\end{theorem}

\begin{proof}
As before, for  $u\in W^{p+1}_1(\Omega)$, there is a $u^I\in
{\mathbb P}_h^*$ so that
\[   \norm u-u^I\norm_{L^1(\Omega)}\le Ch^{p+1}|u|_{ W^{p+1}_1(\Omega) },
\]
where $C>0$ is a constant independent of $u$ and $h$.
Now by the triangle inequality
\begin{eqnarray*}
 \norm u-u^R\norm_{L^1(\Omega)} &\geq&\norm u^I-u^R\norm_{L^1(\Omega)}- \norm u-u^I\norm_{L^1(\Omega)}\\
 &\geq& \sum_{i=1}^{N_e^\circ} \norm u^I-u^R\norm_{L^1(\Omega_i^*)}-C_2h^{p+1}|u|_{ W^{p+1}_1(\Omega)}.
\end{eqnarray*}
By the standard scaling argument, and (\ref{2.4}) or the argument leading to it, we have
\[\norm u^I-u^R\norm_{L^1(\Omega_i^*)}\geq C_2 h^{n/2}_{\min}\norm u^I-u^R\norm_{L^2(\Omega_i^*)}\geq C_2 h^n_{\min} h^{p+1}\sqrt{Q(\tilde D_i^0,\tilde D_i^1,\cdots, \tilde D_i^p)}.
\]
Now invoking (\ref{cs}) and taking minimum of the constants completes proof for the two dimensional case.
\end{proof}

\begin{theorem}\label{thm2.5}
Suppose \[ H_\infty:\quad u\in W_\infty^{p+1}(\Omega),\,
\Omega\subset {\mathbb R}^n,\] where $n=2$. In addition, H1 of Theorem
\ref{thm2.3} holds.
 Then there exists a constant
$C_\infty>0$, independent of $h,u$ and $u^R$, such that
\begin{equation}\label{above1}
 \norm u-u^R\norm_{L^\infty(\Omega)}\geq C_\infty
h^{p+1}\left[(\frac{h_{\min}}{h})^{n/2} \max_{1\leq i\leq N^\circ_e}
\norm\tilde D_i\norm -|u|_{W^{p+1}_\infty(\Omega)}\right],
\end{equation}
 where $n=2$ and  $\tilde D_i$ has components
\[\tilde D_i^{\alpha}=J_i^{(\alpha)}/(h^{p+1}h^{-|\alpha|}_{\min}),\qquad |\alpha|=k,\,\, 0\leq k\leq p.
\]

Suppose assumptions H0 and A1 of Theorem \ref{thm3D} hold for $n=3$. Then (\ref{above1}) holds.  Now $x_i$ is the
face center and $N_e^\circ$ replaced by $N_f^\circ$.
\end{theorem}
\begin{proof}
 Since $u\in W^{p+1}_\infty(\Omega)$, there exists
a $u^I\in \ps_h^*$ such that
\[   \norm u-u^I\norm_{L^\infty(\Omega)}\le C_1h^{p+1}|u|_{W^{p+1}_\infty(\Omega)}.
\]
Now by the triangle inequality
\begin{eqnarray*}
 \norm u-u^R\norm_{L^\infty(\Omega)} \geq\norm u^I-u^R\norm_{L^\infty(\Omega)}- \norm u-u^I\norm_{L^\infty(\Omega)}\\
 \geq \max_{i} \norm u^I-u^R\norm^2_{L^\infty(\Omega_i^*)}-C_2h^{p+1}|u|_{ W^{p+1}_\infty(\Omega)}.
\end{eqnarray*}
Let $U_i=\norm u^I-u^R\norm_{L^\infty(\Omega_i^*)}$ and use (\ref{2.4}) to derive
\begin{eqnarray*}
U_i^2&=&\frac {1}{|\Omega_i^*|}\int_{\Omega_i^*} U_i^2dx\geq \frac {1}{|\Omega_i^*|}\int_{\Omega_i^*}(u^I-u^R)^2dx\\
&\geq& Ch^{-n}\norm u^I-u^R\norm^2_{L^2(\Omega_i^*)}\geq
(h_{\min}/{h})^nh^{2p+2}Q(\tilde D_i^0,\tilde D_i^1,\cdots, \tilde D_i^p).
\end{eqnarray*}
Invoking (\ref{cs}) and taking a common minimum constant completes the proof.
\end{proof}
 Now we impose quasi-uniform conditions on the meshes to get the next theorem.
\begin{theorem}\label{main2}
Suppose that
\begin{itemize}
\item[H0.] $u\in W_s^{p+1}(\Omega)$, $\Omega\subset {\mathbb
R}^n,\quad n=2$ .\\
Suppose the following assumption holds.
\item[H1.]  $u^R \in \ps_p^h$ on a quasi-uniform family of subdivisions $\{{\mathcal T}_h\}$ by triangles or by quadrilaterals,
 i.e., $u^R\in {\mathbb P}_p$ on each element.
 \end{itemize}
 Then
 \begin{itemize}
 \item[(i)]  in case $s=1,2$, there exists a positive constant $C_1$
independent of $h,u$ and $u^R$ such that
\begin{equation}\label{sample} \norm u-u^R\norm_{L^s(\Omega)}\geq C_1h^{p+1}\left(
\left(\sum_{i=1}^{N^\circ_e}h^n\norm
D_i\norm^s\right)^{1/s}-|u|_{W^{p+1}_s(\Omega)}\right),n=2,
\end{equation}

\item[(ii)] in case $s=\infty$, there exists a constant
$C_\infty>0$, independent of $h,u$ and $u^R$, such that
\begin{equation}\notag
 \norm u-u^R\norm_{L^\infty(\Omega)}\geq C_\infty
h^{p+1}\left[\max_{1\leq i\leq N^\circ_e} \norm D_i\norm
-|u|_{W^{p+1}_\infty(\Omega)}\right],
\end{equation}
\end{itemize}
where the components of $D_i$ are
\[ D_i^{\alpha}=
J_i^{(\alpha)}/(h^{p+1-|\alpha|}),\quad
J^{(\alpha)}_i=\llbracket\partial^\alpha u^R\rrbracket_{x_i},\quad   |\alpha|=k,\,\, 0\leq k\leq p.\]

Suppose assumptions H0 and A1 of Theorem \ref{thm3D} hold for $n=3$. Then assertions (i) and (ii) hold. Now $x_i$ is a
face center and $N_e^\circ$ replaced by $N_f^\circ$.
\end{theorem}
\begin{proof}
Note that since $\tilde D_i^{\alpha}= D_i^{\alpha}(\frac {h_{\min}}{h})^k$,
$ \norm{\tilde D}_i\norm^2=
 \sum_{k=0}^p (D^{\alpha}_i)^2 (\frac {h_{\min}}{h})^{2k}.
$ By quasi-uniformness,  $h/h_{\min}$ is uniformly bounded above and
we can replace all occurrences of $h_{\min}$ by $Ch$ in all the
previous theorems. This completes the proof.
\end{proof}
\section{Optimal order convergence implies numerical smoothness}
We present this section independently from other sections so that the reader can read it directly. There might be some
repetitions of the already introduced notations.

\begin{definition}
{\bf Type A Numerical Smoothness.}\label{smooth_def}
Let $u_h\in \ps_p^h$ be a piecewise polynomial of degree $\leq p$ with respect
to a family of subdivisions $\{{\mathcal T}_h\}$ on $\Omega\subset {\mathbb
R}^n$ by $n$-simplices and the alikes
(quadrilateral, hexahedrons etc.). Let $\{x_i\}_{i=1}^{N^\circ}$ be the set of all midpoints of
interior edges for $n=2$ and barycenters of interior faces for $n=3$. Then
$u_h$ is said to be $W^{p+1}_s(\Omega)$-smooth of Type A, $s\geq 1$, if there is
a constant $C_s$, independent of $h$ and $u_h$, such that
\begin{equation}\label{0}
\sum_{i=1}^{N^\circ}h^n\norm D_i\norm^s\leq C_s,
\end{equation}
and $W^{p+1}_\infty(\Omega)$-smooth, if there exists a constant
$C_\infty$ independent of $h$ and $u_h$ such that
\begin{equation}\label{1}
\max_{1\leq i\leq N^\circ} \norm D_i\norm\leq C_\infty,
\end{equation}
 where the components of $D_i$
are the scaled jumps of partial derivatives
\[ D_i^{\alpha}=
J_i^{(\alpha)}/(h^{p+1-|\alpha|}),\quad
J^{(\alpha)}_i=\llbracket\partial^\alpha u_h\rrbracket_{x_i},\quad
\,\, |\alpha|=k,\,\, 0\leq k\leq p.\]
\end{definition}
The A in Type A stands for (a)cross the interface as opposed to the Type I (interior) smoothness below.
As pointed out before, Definition \ref{smooth_def} of numerical smoothness was
first introduced in \cite{SUN} for $n=1$.  It is also worth noting that for the $n=1$ case, several natural conditions for optimal convergence are
already included. These include that the scaled functional value $|D_i^0|\leq C$ for all $i$ in the case of $k=0$, and at the other end in the case of $k=p$ that $|D_i^p|\leq C$ or (\ref{0}) with $s=1$  implies the piecewise constant function $\frac{d^pu_h}{dx^p}$ has bounded variation.

  Intuitively, the smoothness of a numerical solution $u_h\in \ps_p^h$
 should be measured by the boundedness of partial derivatives $\partial^\alpha u_h$. On an element $T\in {\mathcal T}_h$,
 by Taylor expansion around any point $x_{m}$ in $T$, e.g., the center of $T$ or a point on the boundary of $T$ using one-sided
  derivatives, we see that  the quantities $\partial^\alpha u_h(x_{m})$ would be sufficient to
 give information on the interior smoothness. In other words, for this part of smoothness we need a constant $M>0$, independent of $h$, such that
\begin{equation}\label{inner}
  M^\alpha_{m}:=\left\vert\partial^\alpha u_h(x_{m})\right\vert\leq M,\quad  \forall\,\, |\alpha|=k,\,\, 0\leq k\leq p.
   \end{equation}
   On the other hand, the smoothness across the interface boundary of an element, by common sense, should be measured by the jumps of partials $J^{(\alpha)}_i$. The crucial part of Definition
   \ref{smooth_def} is to point out that this intuition needs to be adjusted and that the quantities   $D_i^{\alpha}$ are what is needed to correctly measure numerical
   smoothness across the interface. Notice that the definition does not refer to any convergence to a target solution $u$. In an attempt to
   give a corresponding numerical smoothness of Type I (I for interior) we replace the $D_i$ by $F_i$, which is the difference in the derivatives of $u_h$ and $u$   at $x_m$.
   \begin{definition}
{\bf Type I Numerical Smoothness.}\label{smooth_def1}
Let $u\in C^{p+1}(\Omega),\Omega\subset {\mathbb R}^n$ and let $u_h\in \ps_p^h$ be a piecewise polynomial of degree $\leq p$ with respect
to a family of subdivisions $\{{\mathcal T}_h\}$ of $\Omega$. Let $\{x_i\}_{i=1}^{N^{\mathcal T}}$ be a collection of points $x_i\in T_i\in {\mathcal T}_h,
1\leq i\leq N^{\mathcal T},$ where $N^{\mathcal T}$ is the cardinality of ${\mathcal T}_h$. Then
$u_h$ is said to be $W^{p+1}_s(\Omega)$-smooth of Type I, $s\geq 1$, if there is
a constant $C_s$, independent of $h$ and $u_h$, such that
\begin{equation}
\sum_{i=1}^{N^{\mathcal T}}h^n\norm F_i\norm^s\leq C_s,
\end{equation}
and $W^{p+1}_\infty(\Omega)$-smooth of Type I, if there exists a constant
$C_\infty$ independent of $h$ and $u_h$ such that
\begin{equation}
\max_{1\leq i\leq N^{\mathcal T}} \norm F_i\norm\leq C_\infty,
\end{equation}
 where the components of $F_i$
are the scaled differences between partial derivatives
\[ F_i^{\alpha}=
\partial^\alpha(u-u_h)(x_i)/(h^{p+1-|\alpha|}),\quad
\,\, |\alpha|=k,\,\, 0\leq k\leq p.\]
\end{definition}

   Now in creating a sound smoothness indicator that can account for interior and boundary smoothness, ideally one must incorporate
    the $F^{\alpha}$ (more practically, the computable $M^\alpha=\partial^\alpha u_h$) and $D^\alpha$ quantities. In other words, in an adaptive algorithm, a computable bound ideally should include all or some of them in a proper expression, so long as the cost effect is not too much of a concern. On the other hand, our next theorem concerns a necessary condition for convergence. So we only use the $D^\alpha$ quantities to describe it. The one using $F^\alpha$ will come after that.
   In this perspective we can say that these theorems put the statement
 ``a numerical approximate solution ought to be as
smooth as its targeted exact solution.'' on a rigorous footing.

\begin{theorem}\label{smoothness}
Suppose that $u\in W^{p+1}_s(\Omega),s=1,2,\infty,\Omega\subset
{\mathbb R}^n$ and that $u^R$ is in $\ps_p^h$ on a quasi-uniform family of meshes on $\Omega$, be it made of triangles or
quadrilaterals ($n=2$) or tetrahedrons or hexahedrons ($n=3$). Then a necessary condition for
\[  \norm u-u^R\norm_{L^s(\Omega)}\leq{\mathcal O}(h^{p+1})\]
is for $u^R$ to be $W^{p+1}_s$ smooth. In particular, for
\[  \norm u-u^R\norm_{L^\infty(\Omega)}\leq{\mathcal O}(h^{p+1})\]
a necessary condition is that all jumps in the $k^{th}$ partial
derivatives at midpoints  $x_i$ (n=2) and  face centers $x_i$ ($n=3$)
satisfy
\[  \left\vert\llbracket\partial^\alpha
u^R\rrbracket_{x_i}\right \vert={\mathcal O}(h^{p+1-k}),\quad  |\alpha|=k,\,\, 0\leq k\leq p.\]
Here all smoothness refers to Type A smoothness.
\end{theorem}
\begin{proof} Suppose  $\norm u-u^R\norm_{L^s(\Omega)}\leq Ch^{p+1+\sigma},\sigma\ge 0$. Applying this to inequality (\ref{sample}) deduces the result. Other assertions follow in a similar way.
\end{proof}

Note that all $D_i^\alpha$ need to be bounded for convergence as a consequence of this theorem.

\begin{theorem}\label{smoothness1}
Suppose that $u\in C^{p+1}(\Omega),s=1,2,\infty,\Omega\subset
{\mathbb R}^n, 1\leq n\leq 3$ and that $u^R$ is in $\ps_p^h$ on a quasi-uniform family $\{{\mathcal T}_h\}$ of meshes on $\Omega$, be it made of triangles or
quadrilaterals ($n=2$) or tetrahedrons or hexahedrons ($n=3$). Then a necessary condition for
\[  \norm u-u^R\norm_{L^s(\Omega)}={\mathcal O}(h^{p+1})\]
is for $u^R$ to be $W^{p+1}_s$ smooth. In particular, for
\[  \norm u-u^R\norm_{L^\infty(\Omega)}={\mathcal O}(h^{p+1})\]
a necessary condition is that all the $k^{th}$ partial
derivatives $x_i\in T$ satisfy
\begin{equation}\label{bdd}  |\partial^\alpha u^R(x_i)|={\mathcal O}(1),\quad  |\alpha|=k,\,\, 0\leq k\leq p.
\end{equation}
Here all smoothness refers to Type I smoothness and $\{x_i\}$ is any collection of points, one from each element.
\end{theorem}
\begin{proof}
Since there is no essential difference between the proof for 1-D and those for higher dimensions, we will just give a 1-D version.
Let $\mathcal T_h=\cup_{i=0}^N [x_i,x_{i+1}]$ be a quasi-uniform subdivision on $\Omega=(a,b)$, and let $u\in W^{p+1}_\infty(\Omega)$ and $u^I\in C(\Omega)\cap {\mathbb P}_p^h$ such that $u^I$ restricted to $T_i=(x_i,x_{i+1})$
is the Lagrange nodal interpolant of degree $\leq p$. Let $u^R\in {\mathbb P}_p^h$ be given and
to simplify the presentation, $u_R^{(k)}=\frac{d^k}{dx^k}u^R$, $ u_I^{(k)}=\frac{d^k}{dx^k}u^I$, and $u^{(k)}=\frac{d^k}{dx^k}u$.
At a typical point $x_m\in T_i$, we have the difference in derivatives
\begin{eqnarray}
|\tilde F_i^{(k)}|&:=&|u_R^{(k)}(x_m)-u^{(k)}(x_m)|\notag\\
&\leq& |u_R^{(k)}(x_m)-u_I^{(k)}(x_m)|+|u_I^{(k)}(x_m)-u^{(k)}(x_m)|\notag\\
&=&I_1+I_2.\notag
\end{eqnarray}
On the one hand
\begin{equation}\label{M11}
I_2\leq Ch^{p+1-k}|u|_{W^{{p+1}}_\infty(T_i)},
\end{equation}
and on the other hand
\begin{equation}\label{M22}
I_1\leq Ch^{-k}\norm u^R-u^I\norm_{L^\infty(T_i)},
\end{equation}
where we have used quasi-uniformness of the mesh.
In addition
\begin{eqnarray*}
\norm u^R-u^I\norm_{L^\infty(T_i)}\leq \norm u^R-u\norm_{L^\infty(T_i)}+\norm u^I-u\norm_{L^\infty(T_i)}.
\end{eqnarray*}

Combining all the related estimates, we have

\begin{equation}\label{common1} |\tilde F_i^{(k)}|\leq Ch^{-k}\left(h^{p+1}|u|_{W^{p+1}_\infty(\Omega)}
+\norm u_R-u\norm_{L^\infty(\Omega)}\right),
\end{equation}
which stated in a more practical manner is (\ref{bdd}).

As for the $W^{p+1}_s,s=1,2$ smoothness estimates, we proceed as before, but now
\[   I_2\leq Ch^{p+1-k}|u|_{W^{p+1}_s(T_i)}.\]
Using a standard scaling argument on (\ref{M22}), we have
\[ I_1\leq Ch^{-k}h^{-\frac 1 s}\norm u^R-u^I\norm_{L^s(T_i)}.\]
and
\begin{eqnarray*}
\norm u^R-u^I\norm_{L^s(T_i)}\leq \norm u^R-u\norm_{L^s(T_i)}+\norm u^I-u\norm_{L^s(T_i)},
\end{eqnarray*}
Thus
\begin{eqnarray}
 |\tilde F_i^k|&\leq& Ch^{-k} h^{p+1}|u|_{W^{p+1}_s(T_i)}\notag\\
 &&+Ch^{-k}h^{-\frac 1 s}
\left(\norm u^R-u\norm_{L^s(T_i)}+\norm u^I-u\norm_{L^s(T_i)}\right).
\end{eqnarray}

Combining all the related estimates, we have

\begin{eqnarray*} |\tilde F_i^k|&\leq Ch^{-k}\left(h^{p+1}(1+h^{-\frac 1 s})|u|_{W^{p+1}_s(T_i)}
+h^{-\frac 1 s}\norm u^R-u\norm_{L^s(T_i)}\right).\\
\end{eqnarray*}
Hence
\[|F_i^k|\leq C\left(1+h^{-1/s}|u|_{W^{p+1}_s(T_i)}+h^{-(p+1+1/s)}\norm u^R-u\norm_{L^s(T_i)}\right).
\]
Summing appropriately completes the completes the proof.
\end{proof}
\begin{remark} An immediate consequence of this theorem is that if a single $D_i^\alpha$ is bigger than ${\mathcal O}(1/h)$ , then
$\norm u-u^R\norm_{L^2}$ must be bigger than ${\mathcal O}(h^{p+1})$. Note also that the boundedness of the computable $M^{\alpha}=\partial^\alpha u^R(x_i)$ is necessary. Again it advocates
that $M^\alpha$ should be incorporated one way or another in a computable bound of an adaptive algorithm. The success of using
smoothness indicator of the type $S=(D, M)$ can be found in \cite{SUN2, SR}. It insisted in the control that $D$ and $M$ must be bounded and our two theorems justify that approach.
\end{remark}

Several concluding remarks concerning future extensions are in order here. Our results here pertain to a coordinate free
approach in the sense that our approximating functions $u_h$ are polynomial on each element. The case of the pullbacks being polynomials works as well,
but it is more subtle to handle due to some complications including some tensor product  polynomial approximation issues on quadrilateral meshes ( Arnold {\it et al.} \cite{ABF}). We will report it in a separate paper.

\section{Appendix: $W^{p+1}_s(\Omega)$ smoothness estimates for 1-D}

In this section, we give a short proof of  Theorem \ref{smoothness} for $n=1$. The $s=\infty$ case was communicated to me by T. Sun \cite{SUN3}.

\subsection{Limitation of a short proof in 1-D}
\begin{proof}
Let $\{\mathcal T\}_h=\cup_{i=0}^N [x_i,x_{i+1}]$ be a quasi-uniform subdivision on $\Omega=(a,b)$, and let $u\in W^{p+1}_\infty(\Omega)$ and $u^I\in C(\Omega)\cap {\mathbb P}_p^h$ such that $u^I$ restricted to $(x_i,x_{i+1})$
is the Lagrange nodal interpolant of degree $\leq p$. Let $u^R\in {\mathbb P}_p^h$ be given and
to simplify the presentation, first let $\Omega_i^+=(x_i,x_{i+1})$, $\Omega_i^-=(x_{i-1},x_i)$, $\Omega_i=(x_{i-1},x_{i+1})$, $u_R^{(k)}=\frac{d^k}{dx^k}u^R$
and $ u_I^{(k)}=\frac{d^k}{dx^k}u^I$.

Now
\begin{eqnarray}
|J_i^{(k)}|:=&|u_R^{(k)}(x_i^+)-u_R^{(k)}(x^-_i)|\notag\\
\leq& |u_R^{(k)}(x_i^+)-u^{(k)}(x_i)|+|u_R^{(k)}(x^-_i)-u^{(k)}(x_i)|\label{notd}\\
\leq& |u_R^{(k)}(x_i^+)-u_I^{(k)}(x_i^+)|+|u_I^{(k)}(x_i^+)-u^{(k)}(x_i)|\notag\\
&\,+|u^{(k)}_R(x^-_i)-u^{(k)}_I(x^-_i)|+|u_I^{(k)}(x^-_i)-u^{(k)}(x_i)|\notag\\
=&I_1+I_2+I_3+I_4.\notag
\end{eqnarray}
We have
\begin{equation}\label{M1}
I_2+I_4\leq Ch^{p+1-k}|u|_{W^{{p+1}}_\infty(\Omega)},
\end{equation}
and
\begin{equation}\label{M2}
I_1+I_3\leq Ch^{-k}\left(\norm u^R-u^I\norm_{L^\infty(\Omega_i^+)}+\norm u^R-u^I\norm _{L^\infty(\Omega_i^-)}\right),
\end{equation}
where we have used quasi-uniformness of the mesh.
In addition
\begin{eqnarray*}
\norm u^R-u^I\norm_{L^\infty(\Lambda)}\leq \norm u^R-u\norm_{L^\infty(\Lambda)}+\norm u^I-u\norm_{L^\infty(\Lambda)},\quad \Lambda=\Omega_i^{\pm}.
\end{eqnarray*}

Combining all the related estimates, we have

\begin{equation}\label{common} |J_i^{(k)}|\leq Ch^{-k}\left(h^{p+1}|u|_{W^{p+1}_\infty(\Omega)}
+\norm u_R-u\norm_{L^\infty(\Omega)}\right).
\end{equation}
This completes the proof.
\end{proof}

This proof works even for higher dimensions. In fact, recalling  $W^1_\infty(\Lambda)={\mathcal Lip}(\Lambda), convex\, \Lambda\subset {\mathbb R}^n,n\geq 1$ (\cite{BS}, p 33), we see that
$u\in W^{p+1}_\infty(\Omega)$ implies that the term $\partial ^\alpha u(x_i), |\alpha|=k,0\leq k\leq p$ in (\ref{notd}) makes sense for higher dimensions
and the proof carries over with proper adjustment on the choice of $u^I$..

It is easy to see similar results can be obtained for  $W^{p+1}_s$-smoothness estimates, $s=1,2$ in 1-D case. 

\begin{remark}\label{A.1}
 Unfortunately, this simple proof cannot be
generalized to higher dimensions for several reasons. First, since
we used point value $u^{(k)}(x_i)$ in (\ref{notd}), and by the
Sobolev imbedding theorem \cite{TJ} for a $W^{p+1}_s$ function this
would require $(p+1-k)s\geq n$ for $s=1$ and $(p+1-k)s> n$ for
$s=2$. That means for $s=1$ we must conclude that in general if
$p\geq k> (p+1-n)$ the proof cannot be carried over to higher
dimensions. For $s=2$, the proof cannot be used if $p\geq k\geq
(p+1-n/2)$. For $n=2,3$ this involves derivatives of order $p$ or
$p-1$. Second, the existence of $u^I$ depends on point values and in
high dimensions the mesh point $x_i$ is replaced by a center of an $n-1$
simplex. The local geometry is different and it is hard to find or awkward to describe such an
approximation $u^I$ based on interpolation on the dual mesh. This
problem is overcome by the use of covolumes in our main approach and
point value based interpolants are replaced by projection type
interpolants when necessary.
\end{remark}

\end{document}